\documentclass{amsart}
\usepackage{eurosym}
\usepackage[breaklinks, colorlinks,linkcolor=blue,citecolor=blue,
            urlcolor=blue]{hyperref}
\usepackage{amssymb}

\numberwithin{equation}{section}
\newtheorem{theorem}{Theorem}
\numberwithin{theorem}{section}

\newtheorem{cor}[theorem]{Corollary}

\newtheorem*{thm a}{Theorem A}
\newtheorem*{thm b}{Theorem B}
\newtheorem*{thm c}{Theorem C}

\begin{document}

\title{ Invariance under bounded analytic functions}
\author{ Ajay Kumar}
\address{Department of Mathematics, University of Delhi, Delhi (India) 110007}
\email{nbkdev@gmail.com}
\thanks{The research of the first author is supported by the Junior Research Fellowship of the Council of Scientific and Industrial Research, India (Grant no. 09/045(1232)/ 2012-EMR-I)}

\author{Niteesh Sahni}
\address{Department of Mathematics, Shiv Nadar University, Dadri, Uttar Pradesh (India) 201314}
\email{niteeshsahni@gmail.com}
\thanks{}

\author{Dinesh Singh}
\address{Department of Mathematics, University of Delhi, Delhi (India) 110007}
\email{dineshsingh1@gmail.com}

\subjclass[2010]{Primary 47B37; Secondary 47A25}


\dedicatory{}


\keywords{invariant subspace, inner function, uniform algebra, compact abelian group, multiplier algebra of $BMOA$}

\begin{abstract}
In a recent paper, M. Raghupathi has extended the famous theorem of Beurling to the context of subspaces that are invariant under the class of subalgebras of $H^\infty$ of the form $IH^\infty$, where $I$ is an inner function. In this paper, we provide analouges of the above mentioned $IH^\infty$ related extension of Beurling's theorem to the context of uniform algebras, on compact abelian groups with ordered duals, the Lebesgue space on the real line and in the setting of the space $BMOA$. We also provide a significant simplification of the proof  of the Beurling's theorem in the setting of uniform algebras and a new proof of the Helson-Lowdenslager theorem that generalizes Beurling's theorem in the context of compact abelian groups with ordered duals.
\end{abstract}

\maketitle
\tableofcontents

\section{Introduction and statement of main theorem (Theorem C)} \label{intro}
The results carried in this article stem from the famous and fundamental theorem of Beurling, \cite{brlg}, related to the characterization of the invariant subspaces of the operator of multiplication by the coordinate function $z$-also known as the shift operator-on the classical Hardy space $H^{2}$ of the open unit disk. This invariance is also equivalent to invariance under multiplication by each element of the Banach algebra $H^{\infty }$ of bounded analytic functions on the disk, see \cite[Lemma, p. 106]{hoffman}. The impetus for this article is the recent extension (on the open disk) of Beurling's theorem to the problem of characterizing invariant subspaces on $H^{2}$ where the invariance is under the context of multiplication by each element of the subalgebra $IH^{\infty }$ of $H^{\infty }$ where $I$ is any inner function i.e. $I$ has absolute value $1$ almost everywhere on the boundary $\mathbb{T}$ of the open unit disk. Such an extension has had important applications to interpolation problems and related issues for which we refer to \cite{balnman}, \cite{brnski}, \cite{drtnpkrg}, \cite{hamnrgp}, \cite{jobonho}, \cite{ juknnmc} and \cite{knese}. 

Our principal objective in this paper is to prove versions of the above mentioned extension of Beurling's theorem in the setting of the Hardy spaces on uniform algebras, on compact abelian groups, on the real line and in the context of $BMOA.$ Along the way we first present a new, much simplified and elementary proof of Beurling's theorem on uniform algebras \cite[p. 131]{twgm}. We do this by eliminating, in the context of the Hardy spaces of uniform algebras, the use of a deep result of Kolmogoroff's on the weak 1-1 nature of the conjugation operator and also by eliminating the complicated technicalities of uniform integrability. Later on, in section \ref{CAG}, we also present a new proof of the Helson-Lowdenslager version of Beurling's theorem on compact abelian groups \cite{HCAG}. 

With the purpose of making things clearer we state below Beurling's theorem on the open unit disk and two other connected theorems that are relevant to the rest of this paper. All three theorems below are in the setting of the Hardy spaces of the open unit disk. At appropriate places we shall state the relevant versions of these theorems in the context of various Hardy spaces (mentioned above) and on $BMOA$. Our key objective is to show in the rest of the paper that the Theorem C below has valid versions in various Hardy spaces and on BMOA. It is this theorem that has been proved important in interpolation problems of the open disk.

Let $\mathbb{D}$ denote the open unit disk and let $\mathbb{T}$ be the unit circle in the complex plane $\mathbb{C}$. We use $H^{p}(\mathbb{D})$, $1\leq p<\infty $ to denote the classical Hardy space of analytic functions inside the unit disk $\mathbb{D}$ and $H^{\infty }(\mathbb{D})$ is the space of bounded analytic functions on $\mathbb{D}$. For $1\leq p \leq \infty $, $L^{p}$ denotes the Lebesgue space on the unit circle $\mathbb{T}$ and $H^{p}$ stands for the closed subspace of $L^{p}$ which consists of the functions in $L^{p}$ whose Fourier coefficients for the negative indices are zero. Due to the fact that there is an isometric isomorphism between $H^{p}(\mathbb{D)}$ and $H^{p}$, on certain occasions we will identify $H^{p}(\mathbb{D})$ with $H^{p}$ without comment (see \cite{hoffman}).

The shift operator $S$ on the Hardy space $H^2$, as mentioned above, is defined as $(Sf)(z) = zf(z)$, for all $z \in \mathbb{T}$ and all $f$ in $H^2$. The same definition extends to all Hardy spaces and $S$ is an isometry on all of them. In fact, the operator $S$ is well defined on the larger Lebesgue spaces $L^p$ of which the Hardy spaces are closed subspaces and it is an isometry here as well. The space $L^2$ is a Hilbert space under the inner product 
\[
\langle f, g \rangle = \int \limits^{}_{\mathbb{T}}f(z)\overline{g(z)}dm
\]
where $dm$ is the normalized Lebesgue measure. A proper non-trivial closed subspace $\mathcal{M}$ of a Banach space $X$ is said to be invariant under a bounded linear transformation (operator) $T$ acting on $X$ if $T(\mathcal{M}) \subseteq \mathcal{M}$. Invariant subspaces and their characterization play an import role in operator theory and have numerous applications.

\textit{Note:} All further necessary terminology and notation are given within the relevant sections that shall follow. Throughout the text, $clos_{p}$ stands for the closure in $p$-norm (weak-star when $p=\infty $) and $[$ . $]$ for the $span$.

\begin{thm a}[Beurling's Theorem, \cite{brlg}]
\label{thm A} Every non-trivial shift invariant subspace of $H^2$ has the form $\phi H^2$, where $\phi$ is an inner function.
\end{thm a}

\begin{thm b}(Equivalent version of Beurling's Theorem, \cite[Lemma, p. 106]{hoffman}). \label{thm B}
A closed subspace of $H^{2}$ is shift-invariant iff it is invariant under multiplication by every bounded analytic function in $H^{\infty}$.
\end{thm b}

\begin{thm c}(Extension of Beurling's Theorem, \cite[Theorem 3.1]{mrgh}). \label{thm C}
Let $I$ be an inner function and let $\mathcal{M}$ be a subspace of $L^{p}$, $1 \leq p \leq \infty$ that is invariant under $IH^{\infty }$. Either there exists a measurable set $E$ such that $\mathcal{M}=\chi _{E}L^{p}$ or there
exists a unimodular function $\phi $ such that $\phi IH^{p}\subseteq \mathcal{M}\subseteq \phi H^{p}$. In particular, if $p=2$, then there exists a subspace $W\subseteq H^{2}\ominus IH^{2}$ such that $\mathcal{M}=\phi (W\oplus IH^{2})$.
\end{thm c}

\section{A brief preview}\label{preview}
In section \ref{ua}, we present a simplification of the proof of Beurling's theorem and an analouge of Theorem C in the setting of uniform algebras. In section \ref{CAG}, we produce a new and simple proof of
the Helson-Lowdenslager analogue of Beurling's theorem and a version of Theorem C on compact abelian groups with ordered duals. Section \ref{L2R} describes an avatar of Theorem C for the Lebesgue space of the real line. In section \ref{BMOA}, we present an analouge of Theorem C in the setting of
the space $BMOA$.

\section{Theorem C in the setting of uniform algebras}\label{ua}
Let $X$ be a compact Hausdorff space and let $A$ be a uniform algebra in $C(X)$, the algebra of complex valued continuous functions on $X$. Here, by a uniform algebra we mean a closed subalgebra of $C(X)$ which
contains the constant functions and separates the points of $X$, i.e. for any $x, y \in X$,  $x \neq y$, $\exists$ a function $f \in A$ such that $f(x) \neq f(y)$. For a multiplicative linear functional $\varphi $ in the maximal ideal space of $A$, a \emph{representing measure} $m$ for $\varphi $ is a positive measure on $X$ such that $\varphi (f)=\int fdm$, for all $f\in A$. We shall denote the set of all \emph{representing measures} for $\varphi $ by $M_{\varphi }$. Let $W$ be a convex subset of a vector space $V$, an element $x\in W$ is said to be a 
\emph{core point} of $W$ if whenever $y\in V$ such that $x+y\in W$, then for every sufficiently small $\epsilon >0$, $x-\epsilon y\in W$. A \emph{core measure} for $\varphi $ is a measure which is a \emph{core point} of $M_{\varphi }$.

For $1\leq p<\infty $, $L^{p}(dm)$ is the space of functions whose $p$-th power in absolute value is integrable with respect to the representing measure $m$ and $H^{p}(dm)$ is defined to be the closure of $A$ in $L^{p}(dm) $. $L^{\infty }(dm)$ is the space of $m$-essentially bounded functions and $H^{\infty }(dm)$ is the weak-star closure of $A$ in $L^{\infty }(dm)$. Let $A_{0}$ be the subalgebra of $A$ such that $\int fdm=0$, for all $f\in A$. $H_{0}^{p}(dm)$ is the closure of $A_{0}$ in $L^{p}(dm)$. The real annihilator of $A$ in $L_{R}^{p}$, $1\leq p\leq \infty $, is the space $N^{p} $ which consists of functions $w$ in $L_{R}^{p}$ such that $\int wfdm=0$, for all $f\in A$. The conjugate function of a function $f$ in $ReH^{2}(dm)$ is the function $f^{\ast }$ in $ReH_{0}^{2}(dm)$ such that $f+if^{\ast }\in H^{2}(dm)$. The conjugation operator is the real linear operator which sends $f$ to $f^{\ast }$.

We call a function $I$ in $H^{\infty }(dm)$ $inner$ if $|I|=1$ $m$-almost everywhere. A subspace $\mathcal{M}$ of $L^{p}(dm)$ is said to be \emph{invariant} under $A$ if $A\mathcal{M}\subseteq \mathcal{M}$ or equivalently $A_{0}\mathcal{M}\subseteq \mathcal{M}$. We say $\mathcal{M}$ is \emph{simply invariant} if $A_{0}\mathcal{M}$ is not dense in $\mathcal{M}$. We refer to \cite{twgm} for more details.

Our purpose in the theorem given below is to demonstrate that the Theorem 6.1 in \cite{twgm}, which is the key result that essentially characterizes the invariant subspaces on uniform algebras, can actually be proved without the use of Kolmogoroff's theorem on the $\left( L^{p},L^{1}\right) $ boundedness of the conjugation operator $\left( 0<p<1\right) $ as defined above on uniform algebras and used in \cite{twgm} for observing convergence in measure for the conjugate of a sequence of $L^{1}$ functions. We also
eliminate the use of uniform integrability.

\begin{theorem}\label{gm61}
Suppose the set of representing measures for $\varphi $ is finite dimensional and $m$ is a core measure for $\varphi $. Then there is a 1-1 correspondence between invariant subspaces $\mathcal{M}_{p}$ of $L^{p}\left( m\right) $ and closed (weak star closed if $q=\infty $) invariant subspaces $\mathcal{M}_{q}$ of $L^{q}\left( m\right) $ such that $\mathcal{M}_{q}=\mathcal{M}_{p}\cap L^{q}(dm)$ and $\mathcal{M}_{p}$ is the closure in $L^{p}(dm)$ of $\mathcal{M}_{q}$, $\left( 0<p<q\leq \infty \right) $.
\end{theorem}

\begin{proof}
It is enough to consider the case $q=\infty$ since the other values of $q$ will have an identical proof. Let $\mathcal{M}_p$ be an invariant subspace of $L^{p}(dm)$. Put $M=\mathcal{M}_p\cap L^{\infty }(dm)$. By the Krein-Schmulian criterion, $M$ is weak-star closed. We show $\overline{M}$ (in $L^{p}(dm)$) is equal to $\mathcal{M}_p$. Let us choose any $f\in \mathcal{M}_p$. Let $P\left( \left\vert f\right\vert ^{\frac{p}{2}}\right) $ be the projection in $L_{R}^{2}$ of $\left\vert f\right\vert ^{\frac{p}{2}}$ onto $N^{2}$ (note $N^{1}=N^{\infty }=N^{2}$, see \cite[p. 109]{twgm}). So $h=\left\vert f\right\vert ^{\frac{p}{2}}-P\left( \left\vert f\right\vert^{\frac{p}{2}}\right) \in Re H_{0}^{2}$. Let $h_{n}=\exp \left( -\dfrac{(h + i h^{\ast })}{n}\right) $ where $\ast $ denotes the conjugation operator. Then $h_{n}\in H^{\infty }(dm)$ and $h_{n}f\in M\cap L^{\infty }(dm)$. Further, $h_{n}f\rightarrow f$ in $L^{p}(dm)$ since $h_{n}\rightarrow 1$ boundedly and pointwise. This proves $\mathcal{M}_p\cap L^{\infty }(dm)$ is dense in $\mathcal{M}_p$.

Now suppose that $M$ is a weak-star closed invariant subspace of $L^{\infty }(dm)$ and let $\mathcal{M}_p$ be the closure of $M$ in $L^{p}(dm)$. We must show that $\mathcal{M}_p\cap L^{\infty }(dm)=M$. Clearly $\mathcal{M}_p\cap L^{\infty }(dm)$ is weak-star closed in $L^{\infty }(dm)$. Assume that $M\varsubsetneq \mathcal{M}_p\cap L^{\infty }(dm)$. Then there exists $g\in ~^{\perp }M$ such that $g\notin ~^{\perp }\mathcal{M}_p\cap L^{\infty }(dm)$ ($g\in L^{1}(dm)$). We may assume without loss of generality that $g\in L^{\infty }(dm)$. This can be done by considering $g\exp \left( \dfrac{-\left( \left\vert g\right\vert^{\frac{1}{2}}+P\left( \left\vert g\right\vert ^{\frac{1}{2}}\right) \right) - i \left( \left\vert g\right\vert ^{\frac{1}{2}}+P\left(\left\vert g\right\vert ^{\frac{1}{2}}\right) \right) ^{\ast }}{n}\right) $. Then for each $f\in \mathcal{M}_p$, there is a sequence $\left( f_{n}\right) \subset M$ such that
\begin{equation*}
\int\limits_{X}gf_{n}dm \rightarrow \int\limits_{X}gfdm.
\end{equation*}
But $\int\limits_{X}gf_{n}dm=0$ so that $\int\limits_{X}gfdm=0$. This contradiction implies that $\mathcal{M}_p\cap L^{\infty }(dm)=M$.

\end{proof}

Let $I$ be an inner function in $H^{\infty }(dm)$, then $IH^{\infty }(dm)$ is a subalgebra of $H^{\infty }(dm)$. The following theorem is the version of Theorem C in the setting of uniform algebras i.e. we characterize the subspaces of $L^{p}(dm)$, $1\leq p\leq \infty $, which are invariant under $IH^{\infty }(dm)$.

\begin{theorem}\label{unal}
Let $I$ be an inner function and $\mathcal{M}$ be a subspace of $L^p(dm)$, $1\leq p \leq \infty$, invariant under $IH^{\infty}(dm)$ such that $\int\limits^{}_{X} f dm \neq 0$, for some $f$ in $\mathcal{M}$, then 
\[
I.qH^p(dm) \subseteq \mathcal{M} \subseteq qH^p(dm)
\]
where $q$ is a $m$-measurable function such that $|q|=1$ $m$-a.e. When $p=2$, $$\mathcal{M}=q \left(W \oplus IH^p(dm)\right)$$ for some subspace $W$ of $H^2(dm)$.
\end{theorem}

\begin{proof}
Let us take $\mathcal{M}_1 = clos_p[H^\infty(dm)\mathcal{M}]$, where the closure (weak-star when $p=\infty$) is taken in $L^p(dm)$, then
\begin{equation}\label{eqn31}
I\mathcal{M}_1 = I.clos_p [H^\infty(dm)\mathcal{M}]= clos_p[IH^\infty(dm)\mathcal{M}]
\end{equation}

Since $H^\infty(dm)$ is a Banach algebra under $\sup$-norm, $A_0 \mathcal{M}_1 \subseteq \mathcal{M}_1$. This implies $\mathcal{M}_1$ is an invariant subspace of $L^p(dm)$. Also $\mathcal{M} \subseteq \mathcal{M}_1$, so by hypothesis it follows $\int\limits^{}_{X} f dm \neq 0$, for some $f$ in $\mathcal{M}_1$, and hence $A_0 \mathcal{M}_1$ is not dense in $\mathcal{M}_1$. Therefore, $\mathcal{M}_1$ is simply invariant. Thus $\mathcal{M}_1 = q H^p(dm)$ for some $L^\infty(dm)$ function with $|q|=1$. So, (\ref{eqn31}) becomes
\begin{equation*}
I.qH^p(dm) \subseteq \mathcal{M} \subseteq qH^p(dm)
\end{equation*}

When $p=2$, then there exists a subspace $W \subset H^2(dm)$ such that $W \oplus IH^2(dm) \subseteq H^2(dm)$. Then we have
\begin{equation}\label{eqn32}
\mathcal{M} = q(W \oplus IH^2(dm))
\end{equation}


and the proof is complete.
\end{proof}
When $X= \mathbb{T}$ the unit circle, then the algebra $A$ becomes the disk algebra and $L^p(dm)= L^p$, and we obtain the following part of Theorem $3.1$, in \cite{mrgh}, as a corollary.

\begin{cor}
Let $I$ be an inner function and $\mathcal{M}$ be a subspace of $L^p$, $1\leq p \leq \infty$, invariant under $IH^{\infty}$ but not invariant under $H^\infty$, then there exists a unimodular function $q$ such that $I.qH^p \subseteq \mathcal{M} \subseteq qH^p$. When $p=2$, there exists $W \subseteq H^2 \ominus IH^2$ and $\mathcal{M} = q(W \oplus IH^2)$.
\end{cor}
\section{Theorem C for compact abelian groups}\label{CAG}
We use $K$ to denote a compact abelian group dual to a discrete group $\Gamma $ and $\sigma $ to denote the Haar measure on $K$ which is finite and normalized so that $\sigma (K)=1$. For each $\lambda $ in $\Gamma $, let $\chi _{\lambda }$ denote the character on $K$ defined by $\chi_{\lambda }(x)=x(\lambda )$, for all $x$ in $K$. $L^{p}(d\sigma )$, $1 \leq p <\infty $ denotes the space of functions whose $p^{th}$- power in absolute value is integrable on $K$ with respect to the Haar measure $\sigma $. $L^{\infty }(d\sigma )$ is the space of essentially bounded functions w.r.t. the Haar measure $\sigma $. For $p=2$, the space $L^{2}(d\sigma )$ is a Hilbert space with inner product 
\[
\langle f,g\rangle =\int\limits_{K}^{{}}f(x)\overline{g(x)}d\sigma ,\quad
\forall ~f,g\in L^{2}(d\sigma ) 
\]
and the set of characters $\{\chi _{\lambda }\}_{\lambda \in \Gamma }$ forms an orthonormal basis of $L^{2}(d\sigma )$. Every $f$ in $L^{1}(d\sigma )$ has a Fourier series in terms of $\{\chi _{\lambda }\}_{\lambda \in \Gamma }$
i.e. 
\[
f(x)\sim \sum\limits_{\lambda \in \Gamma }^{{}}a_{\lambda }(f)\chi _{\lambda}(x),\quad \text{where}\quad \hspace{1mm}a_{\lambda}(f)=\int\limits_{K}^{{}}f(x)\overline{\chi _{\lambda }(x)}d\sigma . 
\]
Suppose $\Gamma _{+}$ is a semigroup such that $\Gamma $ is the disjoint union $\Gamma _{+}\cup \{0\}\cup \Gamma _{-}$, where $0$ denote the identity element of $\Gamma $ and $\Gamma _{-}=-\Gamma _{+}$. We say the elements of $\Gamma _{+}$ are positive and those of $\Gamma _{-}$ are negative. The group $\Gamma $ induces an order under these conditions. Details can be found in \cite{RFA}.

For a subspace $\mathcal{M}$ of $L^p(d\sigma)$, we set $\mathcal{M}_\lambda = \chi_\lambda .\mathcal{M}$ and $\mathcal{M}_{-} = clos_p \left[ \bigcup \limits^{}_{\lambda > 0} \chi_{\lambda}.\mathcal{M} \right]$. We say a function in $L^2(d\sigma)$ is \textit{analytic} if $a_\lambda(f) = 0$ for all $\lambda < 0$. $H^2(d\sigma)$ is the subspace of $L^2(d\sigma)$ consisting of all analytic functions in $L^2(d\sigma)$. For each $\lambda \in \Gamma$, $\chi_\lambda$ is an isometry on $H^2(d\sigma)$ and the adjoint
operator of $\chi_\lambda$ is
\[
\chi^*_\lambda f(x) = P \chi_{-\lambda}f(x) 
\]
where $P$ is the orthogonal projection of $L^2(d\sigma)$ on $H^2(d\sigma)$.

A closed subspace $\mathcal{M}$ of a Hilbert space $\mathcal{H}$ is said to be an \textit{invariant subspace} under $\{\chi _{\lambda }\}_{\lambda \in \Gamma _{0}}$ if $\chi _{\lambda }\mathcal{M}\subset \mathcal{M}$ for all $ \lambda $ in $\Gamma _{0}$, where $\Gamma _{0}\subseteq \Gamma $ such that $%
\Gamma _{0}\cap \Gamma _{0}^{-1}=\{0\}$ and $\Gamma _{0}\Gamma_{0}^{-1}=\Gamma $. $\mathcal{M}$ is said to be \textit{doubly invariant} if $\chi _{\lambda }\mathcal{M}\subset \mathcal{M}$ and $\chi _{\lambda }^{\ast}\mathcal{M}\subset \mathcal{M}$ for all $\lambda $ in $\Gamma _{0}$, where $\chi _{\lambda }^{\ast }$ denote the adjoint operator of $\chi _{\lambda }$.
We call a semigroup $\{\chi _{\lambda }\}_{\lambda \in \Gamma _{0}}$ of operators $unitary$ if $\chi _{\lambda }$ is a unitary operator for each $\lambda \in \Gamma _{0}$ and $quasi$ $unitary$ if the closure of $[\bigcup\limits_{\lambda \notin \Gamma _{0}^{-1}}^{{}}\chi _{\lambda }(\mathcal{H})]=\mathcal{H}$. A semi group $\{ T_s\}_{s \in \Gamma_0}$ is called $totally~ non$-$unitary$ if for any $doubly~ invariant$ subspace $\mathcal{M}$ for which $\{ T_s |\mathcal{M}\}_{s \in \Gamma_0}$ is $quasi$-$unitary$, we have $\mathcal{M}= \{ 0\}$. 

First we present a new proof of the Helson-Lowdenslager generalization of Beurling's theorem in the setting of compact abelian groups. The statement of this theorem, in \cite{HCAG}, runs as follows:

\begin{theorem} \cite[Theorem 1]{HCAG}. \label{hcag}
Let $\mathcal{M}$ be an invariant subspace larger than $\mathcal{M}_{-}$. Then $\mathcal{M}= q.H^2$, where $q$ is measurable on $K$ and $|q(x)|=1$ almost everywhere. 
\end{theorem}

Our proof relies on the Suciu decomposition for a semigroup of isometries as stated below.

\begin{theorem}(Suciu's Decomposition, \cite[Theorem 2]{ISSG}). \label{Suciu}
Let $\{ T_s\}_{s \in \Gamma_0}$ be a semigroup of isometries on a Hilbert space $\mathcal{H}$. The space $\mathcal{H}$ may be decomposed uniquely in the form
\[
\mathcal{H} = \mathcal{H}_q \oplus \mathcal{H}_t
\] 
in such a way that $\mathcal{H}_q$ and $\mathcal{H}_t$ are doubly invariant subspaces, $\{ T_s |\mathcal{H}_q\}_{s \in \Gamma_0}$ is quasi-unitary and $\{ T_s |\mathcal{H}_t\}_{s \in \Gamma_0}$ is totally non-unitary.
\end{theorem}

\begin{theorem}
\label{suciu} Let $\mathcal{M}$ be a closed subspace of $L^2(d\sigma)$ and $\mathcal{M}_{-} \subsetneq \mathcal{M}$. If $\mathcal{M}$ is invariant under the semigroup of characters $\{ \chi_{\lambda}\}_{\lambda \geq 0}$ (i.e. $\lambda \in \Gamma_+ \cup \{0\}$ ), then 
\[
\mathcal{M} = \varphi H^2(d\sigma) 
\]
where $\varphi$ is a $\sigma$-measurable function and $|\varphi(x)|= 1$ almost everywhere.
\end{theorem}

\begin{proof}
$\mathcal{M}$ is a Hilbert space being a closed subspace of $L^2(d\sigma)$ and each $\chi_{\lambda}$ in the semigroup $\{ \chi_{\lambda}\}_{\lambda \geq 0}$ is an isometry on $\mathcal{M}$. By Theorem \ref{Suciu}, we can write
\begin{equation}\label{eqn42}
\mathcal{M} = \mathcal{L} \oplus \sum \limits^{}_{\lambda \geq 0} \chi_{\lambda}(\mathcal{N})
\end{equation}
where $\mathcal{N}$ is the orthogonal complement of closure of $[\bigcup \limits^{}_{\lambda > 0} \chi_{\lambda}(\mathcal{M})]$ in $L^2(d\sigma)$ and $\mathcal{L}$ is a quasi unitary subspace of $L^2(d\sigma)$. Clearly $\mathcal{N}$ is non-zero otherwise $\mathcal{M}_{-} = \mathcal{M}$.

Let $\varphi$ be an element in $\mathcal{N}$. We claim that $\varphi$ is non-zero almost everywhere. From equation \ref{eqn42}, we have
$$ \langle \chi_{\delta} \varphi, \chi_{\lambda} \varphi \rangle =\int \limits_{K} \chi_{\lambda - \delta} \varphi \bar{\varphi}d\sigma = 0 \quad \text{for all} \hspace{1mm} \delta, \lambda \in \Gamma_+.$$
This means
\[
\int \limits_{K} \chi_{\gamma} |\varphi|^2 d\sigma = 0 \quad \text{for each non-zero} \hspace{1mm} \gamma \in \Gamma.
\]
and thus $\varphi$ is constant almost everywhere. If we choose $\varphi$ such that $\Vert \varphi \Vert = 1$, then $| \varphi |= 1$ a.e.

Next we assert that $\mathcal{N}$ is one dimensional. To see this assume the existence of a $\psi$ in $\mathcal{N}$ which is orthogonal to $\varphi$. Then we have
\begin{align*}
\langle \chi_{\delta} \varphi, \chi_{\lambda} \psi \rangle & = 0 \quad \text{for all} \hspace{1mm} \delta, \lambda \geq 0.
\end{align*}
which implies
\begin{align*} 
\int \limits_{K} \chi_{\delta - \lambda} \varphi \bar{\psi}d\sigma &= 0
\end{align*}
and thus
\[
\int \limits_{K} \chi_{\gamma} \varphi \bar{\psi}d\sigma = 0, \quad \gamma \in \Gamma
\]
Therefore, every Fourier coefficient of $\varphi \bar{\psi}$ is zero and hence $\varphi \bar{\psi}=0$ a.e., which is possible only when $\psi$ is zero almost everywhere, because $\varphi$ is non-vanishing almost everywhere. So $\mathcal{N}$ is one dimensional and equation (\ref{eqn42}) can be written as
\begin{equation}\label{eqn43}
\mathcal{M} = \mathcal{L} \oplus \sum \limits^{}_{\lambda \geq 0} \chi_{\lambda} \varphi
\end{equation}

Since $\mathcal{L}$ is invariant under $\{ \chi_{\lambda}\}_{\lambda \geq 0}$, for any $f$ in $\mathcal{L}$, we have
\begin{align*}
\langle \chi_{\delta} f, \chi_{\lambda} \varphi \rangle & = 0 \quad \text{for all} \hspace{1mm} \delta, \lambda \geq 0.
\end{align*}
A similar computation which we did above shows that $f=0 \hspace{1mm} a.e.$, which in turn implies $\mathcal{L}$ is zero and equation (\ref{eqn43}) becomes
\begin{equation}\label{eqn44}
\mathcal{M} = \sum \limits^{}_{\lambda \geq 0} \chi_{\lambda} \varphi
\end{equation}

Now multiplication by $\varphi$ is isometry on $L^2(d\sigma)$, so equation \ref{eqn44} takes the form
$$ \mathcal{M} = \varphi H^2(d\sigma)$$
which completes the proof.
\end{proof}

Now we present an analouge of Theorem C in the setting of compact abelian groups with ordered duals.

\begin{theorem}\label{th41}
Let $\mathcal{M}$ be a closed subspace of $L^p(d\sigma)$, $1 \leq p \leq \infty$ and $\mathcal{\tilde{M}} = clos_p \left[ \cup^{}_{\lambda \geq 0} \chi_\lambda. \mathcal{M} \right]$. If $\mathcal{\tilde{M}}_{-} \subsetneq \mathcal{\tilde{M}}$ and for a fixed inner function $I$, $\chi_\lambda.I\mathcal{M} \subseteq \mathcal{M}$,  for each $ \lambda \geq 0$, then 
\[
I \varphi H^p(d\sigma) \subseteq \mathcal{M} \subseteq \varphi H^p(d\sigma)
\]
where $\varphi$ is measurable on $K$ and $|\varphi |=1$ $\sigma$-almost everywhere. When $p=2$, there exists a subspace $W \subseteq H^2(d\sigma)$ such that $\mathcal{M} = \varphi(W \oplus IH^2(d\sigma))$.
\end{theorem}
\begin{proof}
Since multiplication by $I$ is an isometry on $L^p(d\sigma)$ and  $\mathcal{M} \subseteq \mathcal{\tilde{M}}$, we have
\begin{align*}
I.\mathcal{\tilde{M}} = I. clos_p \left[ \bigcup \limits^{}_{\lambda \geq 0} \chi_\lambda . \mathcal{M} \right] =  clos_p \left[ \bigcup\limits^{}_{\lambda \geq 0} \chi_\lambda . I \mathcal{M} \right]  \subseteq \mathcal{M}
\end{align*}
and thus we have 
\begin{equation}\label{eqn41}
I.\mathcal{\tilde{M}} \subseteq  \mathcal{M} \subseteq \mathcal{\tilde{M}}
\end{equation}
For $\delta > 0$ in $\Gamma$
\begin{align*}
\bigcup\limits^{}_{\delta > 0} \chi_\delta . \mathcal{\tilde{M}} = \bigcup\limits^{}_{\delta > 0} \chi_\delta .clos_p \left[ \bigcup\limits^{}_{\lambda \geq 0} \chi_\lambda . \mathcal{M} \right]
= clos_p \left[ \bigcup\limits^{}_{\lambda > 0} \chi_\lambda . \mathcal{M} \right]
 \subseteq \mathcal{\tilde{M}}
\end{align*}

Now $\mathcal{\tilde{M}}$ is invariant and $\mathcal{\tilde{M}}$ is larger than $\mathcal{\tilde{M}}_{-}$, by Theorem $1^{\prime}$, \cite[p. 13]{HCAG}, $\mathcal{\tilde{M}} = \varphi H^p(d\sigma)$, where $\varphi$ is a $\sigma$-measurable function and $|\varphi|=1$ ~$\sigma$-a.e. Thus equation (\ref{eqn41}) becomes
\[
I. \varphi H^p(d\sigma) \subseteq \mathcal{M} \subseteq \varphi H^p(d\sigma)
\]
When $p=2$, there exists a closed subspace $V$ of $\mathcal{M}$ such that $$\mathcal{M} = V \oplus I \varphi H^2(d\sigma).$$
But $V \subseteq \mathcal{M} \subseteq H^2(d\sigma)$, so $V = \varphi W$, where $W$ is a closed subspace of $H^2$, because $\varphi$ is unitary. Therefore
\begin{equation}
\mathcal{M} = \varphi(W \oplus IH^2(d\sigma)).
\end{equation}
\end{proof}

When $I = \chi_{\lambda_0}$, for some $\lambda_0$ in $\Gamma_{+}$, we obtain the following as a corollary to Theorem \ref{th41}.
\begin{cor}\label{th42}
Let $\mathcal{M}$ be a closed subspace of $L^p(d\sigma)$, $1 \leq p \leq \infty$ and $\mathcal{\tilde{M}} = clos_p \left[ \cup^{}_{\lambda \geq 0} \chi_\lambda. \mathcal{M} \right]$. If $\mathcal{\tilde{M}}_{-} \subsetneq \mathcal{\tilde{M}}$ and for a fixed positive element $\lambda_0$ in $\Gamma$, $\chi_\lambda.\mathcal{M} \subseteq \mathcal{M}$,  for each $ \lambda \geq \lambda_0$, then 
\[
\chi_{\lambda_0}. \varphi H^p(d\sigma) \subseteq \mathcal{M} \subseteq \varphi H^p(d\sigma)
\]
where $\varphi$ is measurable on $K$ and $|\varphi |=1$ $\sigma$-almost everywhere.
\end{cor}

We observe that Theorem $1.3$, in \cite{DPRS}, becomes a special case of Corollary \ref{th42}, when $p=2$. If we take $\Gamma = \mathbb{Z}$ and $\lambda_0 = 2$, then $\chi_{\lambda_{0}} = z^2$ and $\chi_{\lambda} \mathcal{M} \subseteq \mathcal{M}$, $\forall ~ \lambda \geq \lambda_0$ means invarince under $H^{\infty}_1$.

\begin{cor}\cite[Theorem 1.3]{DPRS}. \label{z2z3}
Let $ \mathcal{M}$ be a norm closed subspace of $L^2$ which is invariant for $H^\infty_1$, but is not invariant for $H^\infty$. Then there exist scalars $\alpha, \beta$ in $\mathbb{C}$ with $|\alpha|^2 + |\beta|^2 = 1$ and $\alpha \neq 0$ and a unimodular function $J$, such that $\mathcal{M}= JH^2_{\alpha \beta}$.
\end{cor}

\section{Theorem C for the Lebesgue space of the real line}

\label{L2R} Let $L^2(\mathbb{R})$ denote the space of square integrable functions on the real line $\mathbb{R}$. We consider $H^2(\mathbb{R})$ a closed subspace of $L^2(\mathbb{R})$ which consists of functions whose Fourier transform 
\[
F(\lambda) = \int \limits^{\infty}_{-\infty}f(x)e^{-i\lambda x}dx 
\]
is zero almost everywhere for every $\lambda < 0$. A subspace $\mathcal{M}$ of $L^2(\mathbb{R})$ is said to be \emph{invariant} if $e^{i \lambda x} \mathcal{M} \subseteq \mathcal{M}$, for all $\lambda > 0$ and \emph{simply invariant} if $e^{i \lambda x}\mathcal{M} \subsetneq \mathcal{M}$, for $\lambda > 0$. If $e^{i \lambda x}\mathcal{M} = \mathcal{M}$ for all real $\lambda$, then we call $\mathcal{M}$ \emph{doubly invariant}. We say a function $I \in H^2(\mathbb{R})$ is inner if $|I(x)|=1$ almost everywhere.

In this section, we give an extension along the lines of \cite{DPRS} and \cite{mrgh} of the Beurling-Lax theorem, \cite[p. 114]{hoffman} for the Lebesgue space $L^2(\mathbb{R})$ of the real line.

\begin{theorem}
\label{l2R} Let $\mathcal{M}$ be a closed subspace of $L^2(\mathbb{R})$. If $I$ is an inner function and $e^{i\lambda x}I\mathcal{M} \subseteq \mathcal{M}$, for all $\lambda \geq 0$, then either there exists a measurable subset $E$ of $\mathbb{R}$ such that $\mathcal{M}=
\chi_E L^2(\mathbb{R})$ or 
\[
\mathcal{M} = q \left( W \oplus I  H^2(\mathbb{R}) \right) 
\]
where $q$ is measurable function on the real line and $|q(x)|=1$ almost everywhere.
\end{theorem}

\begin{proof}
Consider the subspace
\[
\mathcal{N} = clos_2 \left[ \bigcup\limits^{}_{\lambda \geq 0} e^{i\lambda x} \mathcal{M} \right].
\]
Our  consideration of $\mathcal{N}$ implies that $\mathcal{M}$ is a subspace of $\mathcal{N}$ and $e^{i\lambda x} \mathcal{N} \subseteq \mathcal{N}$, for all $\lambda > 0$. Since multiplication by $I$ is an isometry on $L^2(\mathbb{R})$ and $\mathcal{M}$ is a closed subspace of $L^2(\mathbb{R})$, $e^{i \lambda x} I \mathcal{M}  \subseteq \mathcal{M}$, for $\lambda \geq 0$. So 
\begin{align*}
I\mathcal{N}  =  I clos_2 \left[ \bigcup\limits^{}_{\lambda \geq 0} e^{i \lambda x} \mathcal{M} \right] = clos_2 \left[ \bigcup\limits^{}_{\lambda \geq 0} e^{i \lambda x} I\mathcal{M} \right]  \subseteq \mathcal{M}
\end{align*}
Thus we obtain the inclusion
\begin{equation}\label{inc42}
I\mathcal{N} \subseteq  \mathcal{M} \subseteq \mathcal{N}
\end{equation}

If $e^{i\lambda x}\mathcal{N} = \mathcal{N}$, for some $\lambda$ and hence for all $\lambda$, then by \cite[Theorem, p. 114]{hoffman}, $\mathcal{N} = \chi_E L^2(\mathbb{R})$, for some fixed measurable subset $E$ of the real line. Thus, $I\mathcal{N} = \mathcal{N}$ and by the inclusion in (\ref{inc42}), we have $\mathcal{M} = \chi_E L^2(\mathbb{R})$.

 On the other side, if $e^{i\lambda x}\mathcal{N} \subsetneq \mathcal{N}$, then again by \cite[Theorem, p. 114]{hoffman}, $\mathcal{N} = qH^2(\mathbb{R})$, where $q$ is a measurable function on the real line and $|q(x)|=1$ almost everywhere. So 
\[
I. qH^2(\mathbb{R}) \subseteq \mathcal{M} \subseteq qH^2(\mathbb{R})
\] 
Now we see that 
\begin{align*}
\mathcal{M} \ominus I q H^2(\mathbb{R}) & \subseteq qH^2(\mathbb{R}) \ominus I. qH^2(\mathbb{R}) \\
& = q \left( H^2(\mathbb{R}) \ominus I H^2(\mathbb{R}) \right)
\end{align*}
So there exists a subspace $W \subseteq H^2(\mathbb{R}) \ominus I H^2(\mathbb{R})$ such that 
$$q W = \mathcal{M} \ominus I q H^2(\mathbb{R}) $$
or we can write 
$$ \mathcal{M} = q \left( W \oplus I  H^2(\mathbb{R}) \right)$$
This completes the proof.
\end{proof}

\section{Theorem C in the context of \textit{BMOA}}\label{BMOA}
Let $f \in H^1$, then we say that $f \in BMOA$ if 
\[
\sup\limits^{}_{I} \frac{1}{|I|} \int \limits^{}_{I}|f - f_I|d\theta <
\infty 
\]
where $I$ is a subarc of $\mathbb{T}$ and $f_I= \frac{1}{|I|} \int \limits^{}_{I} f d\theta $.

The space $BMOA$ is a Banach space and the dual of $H^1$. The duality is due to a famous theorem of C. Fefferman which we state below.

\textbf{Fefferman's theorem} (disk version), \cite[p. 261]{fefnst}. \emph{ Each $f \in BMOA$ is a linear functional on $H^1$ and its action is given by}
$$ f(g) = \lim \limits^{}_{r \to 1-}  \int \limits^{}_{\mathbb{T}} f(re^{i\theta})\overline{g(re^{i\theta})}d\theta, \quad for ~all ~g \in H^1.$$

This duality induces the weak-star topology on $BMOA$. The weak star closed subspaces of $BMOA$ invariant under the operator of multiplication by the coordinate function $z$ are well known, see \cite{bns1}, \cite{sahsin2} and \cite{dsus}. It is also easy to see that the appropriate version of Theorem B is valid in this context i.e. the shift invariant subspaces are identical to those that are invariant under multiplication by each element of the algebra of multipliers of $BMOA$ which we call as the multiplier algebra of $BMOA$ and will be denoted by $\mathfrak{M}_{\textsc{bmoa}}$. In a way similar to sections \ref{ua}, \ref{CAG} and \ref{L2R}, we characterize the weak-star closed subspaces of $BMOA$ which are invariant under $B(z)\mathfrak{M}_{\textsc{bmoa}}$, where $B(z)$ is a finite Blaschke factor. The reason why we do not use an arbitrary inner function instead of a finite Blaschke factor $B$ is that the only inner functions that multiply $BMOA$ are finite Blaschke factors. The collection $\mathfrak{M}_{\textsc{bmoa}}$ is well known through the work of Stengenga \cite{dstg}. This enables us to present here the appropriate version of Theorem C in the setting of $BMOA$.

Our proof will make use of the following description of an orthonormal basis for $H^2$ in terms of a finite Blaschke factor $B(z)$ of order $n$:
\begin{theorem}[Singh and Thukral, \cite{dsvt}]\label{h2fct}
Let $\alpha_1,\ldots,\alpha_n\in\mathbb{D}$, and $B(z)=\prod \limits_{i=1}^n \frac{z-\alpha_i}{1-\overline{\alpha_i} z}$ be a Blaschke factor of order $n$. We assume that $\alpha_1=0$.
Define $e_{j,0}=\frac{\sqrt{1-|\alpha_j|^2}}{1-\overline{\alpha_j} z} \prod \limits_{i=1}^j \frac{z-\alpha_i}{1-\overline{\alpha_i} z}$. Then $\{e_{j,0}B(z)^m:m=0,1,2,\ldots\}$ is an orthonormal basis for $H^2$.
\end{theorem}

\begin{theorem}
\label{mulalg} Let $B(z)$ be a finite Blaschke factor and $\mathcal{M}$ be a weak-star closed subspace of $BMOA$ which is invariant under $B(z)\mathfrak{M}_{\textsc{bmoa}}$. Then, there exists a finite dimensional subspace $W$ of $BMOA$ and an inner function $\varphi$ such that 
\[
\mathcal{M}=\varphi\left(W\oplus B(z) BMOA\right)\cap BMOA. 
\]
\end{theorem}

\begin{proof}
First, we shall show that $\mathcal{M}$ has non-empty intersection with $H^\infty$. 
Using the fact that $\{e_{j,0}B(z)^m:m=0,1,2,\ldots\}$ is an orthonormal basis, in Theorem \ref{h2fct}, any $f\in\mathcal{M}$ can be written as
\begin{equation}\label{eq53}
f(z)=e_{00}f_0(B(z)) + \cdots + e_{n-1,0} f_{n-1}(B(z)),
\end{equation} 
for some $f_0(z),\ldots,f_{n-1}(z)$ in $H^2$. For $i= 0, 1, \ldots, n-1$, we define functions
$$g^{(i)}(z)=\exp \left(-|f_i(z)|-i |f_i(z)|^\sim \right),$$
where $\sim$ stands for the harmonic conjugate. Consider the function
$$h(z) = g^{(0)}(B(z)) \ldots g^{(n-1)}(B(z)).$$
It is easy to see that $B(z)h(z)f(z) \in H^\infty$. Define $h_t(z)=h(tz)$ for $t \in(0,1)$. For each such fixed $t$, $h_t(z)$ is a multiplier of $BMOA$ (see \cite[Corollary 2.8]{dstg}), so $h_t(z)f(z) \in BMOA$. By Lemma $2$, in \cite{bns2}, $h_t(z)f(z) \in [f]$, where $[f]= \text{wek-star closure of } span \{fp\}$, for all polynomials $p$.  Since $\mathcal{M}$ is weak-star closed in $BMOA$ and invariant under $B(z)\mathfrak{M}_{\textsc{bmoa}}$, so $h_t(z)f(z) \in \mathcal{M}$. Again following the arguments of Lemma $2$, in \cite{bns2}, $h_t(z)f(z)$ converges weak-star to $h(z)f(z)$, so $h(z)f(z)$ also belongs to $\mathcal{M}$, and hence $B(z)h(z)f(z)$ belongs to $\mathcal{M}$. This establishes the claim that $\mathcal{M} \cap H^\infty$ is non-empty.

The space $\mathcal{M} \cap H^\infty$ is a weak-star closed subspace of $H^\infty$ and is invariant under the algebra $BH^\infty$, so by Theorem $3.1$ in \cite{mrgh}, there exists an inner function $\varphi$ such that
\begin{equation}\label{eq51}
\varphi B(z)H^{\infty} \subseteq \mathcal{M} \cap H^{\infty} \subseteq \varphi H^{\infty}.
\end{equation}
It has been established in Theorem $4.1$, in \cite{sahsin2} that $\overline{I H^\infty}=I BMOA\cap BMOA$ for any inner function $I$. Therefore,
\begin{equation}\label{eq52}
\varphi B(z)BMOA\cap BMOA \subseteq \overline{\mathcal{M} \cap H^{\infty}} \subseteq \varphi BMOA\cap BMOA.
\end{equation}
The bar in (\ref{eq52}) denotes weak-star closure in $BMOA$.

We claim that $\overline{\mathcal{M} \cap H^{\infty}}=\mathcal{M}$. Consider the decomposition \ref{eq53} for any $f\in\mathcal{M}$. For each $i= 0, 1, \ldots, n-1$, define a sequence of$H^{\infty}$ functions
$$g^{(i)}_m(z)=\exp \left(\frac{-|f_i(z)|-i |f_i(z)|^\sim}{m} \right).$$
Define
$$O_m(z) = g^{(0)}_m(B(z)) \ldots g^{(n-1)}_m(B(z)).$$
It can be seen that $O_m(z)f(z)\in H^\infty$, and $O_m(z)\to 1~a.e.$.

 As seen above, for each fixed $m$, $O_m(tz)f(z)\in \mathcal{M}$. Also, $O_m(tz)f(z)$ converges weak-star to $O_m(z)f(z)$ in $BMOA$, so $O_m(z)f(z) \in \mathcal{M}$ and hence in $\mathcal{M} \cap H^\infty$. 

Now, $O_m(z)f(z)\to f(z)~a.e.$ and $\Vert O_m(z)f(z) \Vert_{\infty} \leq K$, for some constant $K$. By Dominated Convergence Theorem, for every $\epsilon > 0$,
\[
\int^{}_{\mathbb{T}} |O_m(z)f(z) - f(z)| < \epsilon \text{ for sufficiently large }m.
\]
This means that for each polynomial $p\in H^1$ with upper bound $M_p$, we can find sufficiently large $m, n$ such that
\[
\int^{}_{\mathbb{T}} |O_m(z)f(z) - O_n(z)f(z)| < \frac{\epsilon}{|M_p|}.
\]
So
\begin{align*}
\int^{}_{\mathbb{T}} |O_m(z)f(z)\overline{p(z)} - O_n(z)f(z)\overline{p(z)}| &= \int^{}_{\mathbb{T}} |O_m(z)f(z) - O_n(z)f(z)||\overline{p(z)}|\\
 &< \frac{\epsilon}{|M_p|}.|M_p|=\epsilon.
\end{align*}
Thus $\left( O_m f \right)(p)$ is a Cauchy sequence for each polynomial $p$ in $H^1$. Moreover, $ \Vert O_m(z)f(z) \Vert_{BMOA} \leq K$. By Ex. $13$, in \cite[p. 76]{RFA}, $\{O_m(z)f(z) \}$ converges weak-star to some $h(z)$ in $BMOA$.

We claim that $h(z)$ coincides with $f(z)$. Note that $(O_mf )(k)$ converges weak-star to $h(k)$, for each $k$ in $H^1$. So $(O_mf )(k_{z_0})$ converges weak-star to $h(k_{z_0})$, where  $k_{z_0}= \frac{1}{1-\bar{z_0}z}$ is the reproducing kernel in $H^1$, and $z_0$ is an arbitrary but fixed element of $\mathbb{D}$. Therefore, $O_m(z_0)f(z_0)=(O_mf )(k_{z_0})$ converges weak-star to $h(z_0)=h(k_{z_0})$. Since $z_0$ was arbitrarily chosen, so $O_m(z)f(z)$ converges weak-star to $h(z)$, for each $z \in \mathbb{D}$. But $O_m(z)f(z)$ converges to $f(z)$ a.e., so $h = f$ a.e.

This proves that $O_m f$ converges to $f$ weak-star in $BMOA$, and hence $\overline{\mathcal{M} \cap H^\infty}=\mathcal{M}$. 
The inequality (\ref{eq52}) now reads
\begin{equation}\label{eq54}
\varphi B(z) BMOA \cap BMOA \subset \mathcal{M} \subset\varphi BMOA \cap BMOA.
\end{equation}
Let $\overline{\overline{\mathcal{M}}}$ stand for the closure of $\mathcal{M}$ in $H^2$. Taking closure in $H^2$ throughout (\ref{eq54}) we get
\begin{equation}\label{eq55}
\varphi B(z) H^2 \subset \overline{\overline{\mathcal{M}}} \subset\varphi H^2.
\end{equation}
From (\ref{eq55}), we see that $\overline{\overline{\mathcal{M}}}\ominus\varphi B(z) H^2\subset \varphi (H^2\ominus B(z) H^2)$.
So, there exists a subspace $W_1$ of $H^2\ominus B(z) H^2$ such that $\overline{\overline{\mathcal{M}}}\ominus\varphi B(z) H^2=\varphi W_1$. Moreover, $\dim W_1\le n$. Therefore
$$\overline{\overline{\mathcal{M}}}=\varphi W_1\oplus \varphi B(z) H^2.$$
Since $\mathcal{M} \subset \overline{\overline{M}}$, we have the following form for $\mathcal{M}$:
\begin{equation}\label{eq56}
\mathcal{M} = \varphi W\oplus \varphi B(z) \mathcal{N};
\end{equation}
where $W$ is a subspace of $W_1$, and $\mathcal{N}$ is a subspace of $H^2$.

Now $H^2\ominus B(z)H^2 = \{e_{0,0},e_{1,0},\ldots, e_{n-1,0}\}\subset H^\infty$ and consequently $W\subset H^\infty$. Thus in equation (\ref{eq56}) we have $W\subset BMOA$, and because $\varphi B(z)$ is inner, we also have $\mathcal{N} \subset BMOA$.

In light of (\ref{eq54}) we see that $\varphi B(z) BMOA \cap BMOA \subset \varphi B(z) \mathcal{N}$. But $\mathcal{N} \subset BMOA$. So $\varphi B(z) \mathcal{N} = \varphi B(z) BMOA \cap BMOA$. This completes the proof of the theorem.
\end{proof}
If we take $B(z)=1$, then invariance under $\mathfrak{M}_{\textsc{bmoa}}$ is equivalent to invariance under the operator $S$ of multiplication by coordinate function $z$ on $BMOA$, and the results in \cite[Theorem 3.1]{bns1}, \cite[Theorem 4.3]{sahsin2} and \cite[Theorem C]{dsus} can be derived as corollaries of the above theorem.

\begin{cor}
Let $\mathcal{M}$ be a weak star closed subspace of $BMOA$ invariant under $S$. Then there exists a unique inner function $\varphi$ such that $\mathcal{M}= \varphi BMOA \cap BMOA$.
\end{cor}

Replacing $B(z)$ with $z$, we obtain common invariant subspaces of $S^2$ and $S^3$ and Theorem $3.1$ in \cite{sahsin2} is received as corollary of Theorem \ref{mulalg}.

\begin{cor}
Let $\mathcal{M}$ be weak star closed subspace of $BMOA$ which is invariant under $S^2$ and $S^3$ but not invariant under $S$. Then there exists an inner function $I$ and constants $\alpha, \beta$ such that 
\[
\mathcal{M} = I BMOA_{\alpha, \beta} \cap BMOA. 
\]
\end{cor}

\begin{proof}
The result follows by taking $B(z)= z$ and $W$ as subspace of $span\{ 1, z \}$.
\end{proof}

\end{document}